\documentclass[11pt,letterpaper]{amsart}
\usepackage{mathrsfs}

\usepackage{color}
\usepackage{bm}
\usepackage{amsfonts,amssymb}
\usepackage{dsfont}
\usepackage{amscd}
\usepackage{extarrows}
\usepackage{amsmath}
\usepackage{mathrsfs}
\usepackage{enumerate}
\usepackage{amscd}
\usepackage[all]{xy}
\usepackage{hyperref}
\numberwithin{equation}{section}
\theoremstyle{plain} 

\theoremstyle{definition} 

\newtheorem{thm}{Theorem}[section]

\newtheorem{lem}[thm]{Lemma}
\newtheorem{prop}[thm]{Proposition}

\theoremstyle{definition}
\newtheorem{defn}{Definition}[section]

\theoremstyle{remark}
\newtheorem{rem}{Remark}[section]

\newcommand{\be}{\begin{equation}}
	\newcommand{\ee}{\end{equation}}
\newcommand{\bea}{\begin{eqnarray}}
	\newcommand{\eea}{\end{eqnarray}}
\newcommand{\ben}{\begin{eqnarray*}}
	\newcommand{\een}{\end{eqnarray*}}
\newcommand{\bt}{\begin{split}}
	\newcommand{\et}{\end{split}}
\newcommand{\bet}{\begin{equation}}
	\newcommand{\mc}{\mathbb{C}}

	\newcommand{\ra}{\rightarrow}
	\newcommand{\beq}{\begin{equation*}}
		\newcommand{\eeq}{\end{equation*}}
	\newcommand{\bi}{\begin{itemize}}
		\newcommand{\ei}{\end{itemize}}

\begin{document}
		
\title[Approximation and extension of Hermitian metrics]
{Approximation and extension of Hermitian metrics on holomorphic vector bundles over Stein manifolds}
		
\author[F. Deng]{Fusheng Deng}
\address{Fusheng Deng: \ School of Mathematical Sciences, University of Chinese Academy of Sciences\\ Beijing 100049, P. R. China}
\email{fshdeng@ucas.ac.cn}
\author[J. Ning]{Jiafu Ning}
\address{Jiafu Ning: \ School of Mathematics and Statistics, HNP-LAMA, Central South University, Changsha, Hunan 410083, P. R. China}
\email{jfning@csu.edu.cn}
\author[Z. Wang]{Zhiwei Wang}
\address{ Zhiwei Wang: \ School
	of Mathematical Sciences\\Beijing Normal University\\Beijing\\ 100875\\ P. R. China}
\email{zhiwei@bnu.edu.cn}
\author[X. Zhou]{Xiangyu Zhou}
\address{Xiangyu Zhou: Institute of Mathematics\\Academy of Mathematics and Systems Sciences\\and Hua Loo-Keng Key
    Laboratory of Mathematics\\Chinese Academy of
    Sciences\\Beijing\\100190\\P. R. China}
\address{School of
    Mathematical Sciences, University of Chinese Academy of Sciences,
    Beijing 100049, China}
\email{xyzhou@math.ac.cn}

\begin{abstract}
We show that a singular Hermitian metric on a holomorphic vector bundle over a Stein manifold which is negative in the sense of Griffiths (resp. Nakano)
can be approximated by a sequence of smooth Hermitian metrics with the same curvature negativity.
We also show that a smooth Hermitian metric on a holomorphic vector bundle over a Stein manifold restricted to a submanifold
which is negative in the sense of Griffiths (resp. Nakano) can be extended to the whole bundle with the same curvature negativity.
 \end{abstract}

		\maketitle
		
\section{Introduction}\label{sec:intro}
It is known that any plurisubharmonic function on a Stein manifold can be globally approximated point-wise
by a decreasing sequence of smooth plurisubharmonic functions \cite{For-Nar80}.
For Stein manifolds, another important result states that any plurisubharmonic function
on the submanifold of a Stein manifold can always be extended to a plurisubharmonic function on the ambient space \cite{Sa}.
Note that a plurisubharmonic function on a complex manifold can be viewed as a positively curved singular Hermitian metric on the trivial line bundle over the manifold.
The aim of the present work is to prove similar results for positively curved (in certain sense, to be clarified later) singular Hermitian metrics
on holomorphic vector bundles over Stein manifolds.



To state the main results, we first recall some notions about curvature positivity of singular Hermitian metrics on holomorphic vector bundles.

Let $\pi:E\ra X$ be a holomorphic vector bundle over a complex manifold $X$.
Then a singular Hermitian metric is a measurable section $h$ of $E^*\otimes\overline{E^*}$
that gives a Hermitian inner product on $E_x$, the fiber of $E$ over $x$,  for almost all $x\in X$.
Given such a singular Hermitian metric $h$ on $E$, we can define a dual singular Hermitian metric $h^*$
in the dual bundle $E^*$ of $E$ in a natural way.

The following definition is given in \cite{BP}.

\begin{defn}
Let $\pi:E\ra X$ be a holomorphic vector bundle over a complex manifold $X$.
A singular Hermitian metric $h$ on $E$ is negatively curved in the sense of Griffiths
if $||u||^2_h$ is plurisubharmonic for any local holomorphic
section $u$ of $E$, and we say that $h$ is positively curved in the sense of Griffiths
if the dual metric $h^*$ on the dual bundle $E^*$ of $E$ is negatively curved in the sense of Griffiths.
\end{defn}

For smooth Hermitian metrics on holomorphic vector bundles,
Berndtsson discovered an equivalent characterization in \cite{B}, as follows.
Let $h$ be a smooth Hermitian metric on $E$.
For any local coordinate system $\{z=(z_1,\cdots,z_n),U\}$ on $X$ and any
$n$-tuple local holomorphic sections
  $v=(v_1,\cdots,v_n)$ of $E$ over $U$,  we set
  $$T^h_v:=\sum_{j,k=1}^{n}(v_j,v_k)_h\widehat{dz_j\wedge d\bar{z}_k},$$
  where $\widehat{dz_j\wedge d\bar{z}_k}$ is the wedge product of
all $dz_s$ and $d\bar z_s$ except $dz_j$ and $d\bar z_k$, multiplied by a constant of absolute value 1,
  such that  $$idz_j\wedge d\bar{z}_k\wedge \widehat{dz_j\wedge d\bar{z}_k}=
  idz_1\wedge d\bar{z}_1\wedge\cdots\wedge idz_n\wedge d\bar{z}_n=:dV_z.$$
Then $h$ is negatively curved if and only if if $i\partial\bar\partial T^h_v\geq 0$
for all choices of $v$.

If $h$ is assumed to be singular, $T^h_v$ can also be defined as a current on $U$ of bi-degree $(n-1, n-1)$,
and hence $i\partial\bar\partial T^h_v$ is an $(n,n)$-current on $U$.
Following Berndtsson's observation, Raufi induces the following definition \cite{R}
of Nakano negativity and dual Nakano positivity for singular Hermtian metric on holomorphic vector bundles.

\begin{defn}\label{Defn}
Let $\pi:E\ra X$ be a holomorphic vector bundle over a complex manifold $X$ of dimension $n$.
A singular Hermitian metric $h$ on $E$ is negatively curved in the sense of Nakano,
  if $$i\partial\bar\partial T^h_v\geq 0$$
holds for any $n$-tuple local holomorphic sections $v=(v_1,\cdots, v_n)$, and we say that $h$ is dual Nakano positive
if $h^*$ is negatively curved in the sense of Nakano.
\end{defn}

The first result of the note is the following
\begin{thm}\label{thm c g}
Let $(E,h)$ be a singular hermitian holomorphic vector bundle over a Stein manifold $X$.
Assume $h$ is negatively curved in the sense of Griffiths (resp. Nakano).
Then there exists a sequence of smooth hermitian metrics $\{h_k\}_{k=1}^\infty$ with negative curvature in the sense of Griffiths (resp. Nakano),
such that for any compact subset $K\subset X$, there is $k(K)$, for $k\geq k(K)$, $h_k$ decrease to $h$
pointwise on $K$.
\end{thm}

A local version, namely, the case that $X$ is a polydisc and $E$ is trivial, of the Theorem \ref{thm c g} was proved in \cite{BP} and \cite{R}.

The second result is about extension of Hermitian metrics.

\begin{thm}\label{thm ext G}
Let $E$ be a  holomorphic vector bundle over a Stein manifold $X$, and $Y$ be a closed submanifold of $X$.
Assume that $h$ is a smooth hermitian metric on $E|_Y$ with negative curvature in the sense of Griffiths (resp. Nakano),
then $h$ can be extended to a smooth hermitian metric $\tilde{h}$ on $E$  with negative curvature in the sense of Griffiths (resp. Nakano).
\end{thm}

We do not know if the result in Theorem \ref{thm ext G} still holds if $h$
is assumed to be singular.


\subsection*{Acknowledgements}
		This research is supported by National Key R\&D Program of China (No. 2021YFA1002600 and No. 2021YFA1003100). The authors are partially supported respectively by NSFC grants (11871451,  12071485, 12071035, 11688101).
		The first author is  partially supported by the Fundamental Research Funds for the Central Universities.
		The third author is partially supported by Beijing Natural Science Foundation (1202012, Z190003).

\section{approximation of singular Hermitian metrics}
We put the following obvious results as a lemma.
\begin{lem}\label{add}
Let $E$ be a holomorphic vector bundle over a complex manifold $X$,
 $h_1$ and $h_2$ be two singular hermitian metrics on $E$. Then

 $(1)$ if $h_1$ and $h_2$ are both with negative curvature in the sense of Griffiths, so is $h_1+h_2$;

 $(2)$ if $h_1$ and $h_2$ are both with negative curvature in the sense of Nakano, so is $h_1+h_2$.
\end{lem}

%
%

From the Oka-Grauert principle for holomorphic vector bundles, we have the following

\begin{lem}\label{trivial}
Let $E$ be a holomorphic vector bundle over a Stein manifold $X$.
There is a holomorphic vector bundle $F$ on $X$,
such that $E\oplus F$ is a trivial holomorphic vector bundle.
\end{lem}
\begin{proof}
Let $f:X\ra G(r, N)$ be a continuous map such that $f^{-1}U$ and $E$ are isomorphic as complex topological vector bundles over $X$,
where $G(r, N)$ is a Grassmann manifold and $U\ra G(r,N)$ is the universal vector bundle over $G(r,N)$.
Note that $U$ is a subbundle of the trivial vector bundle $G(r,N)\times \mc^N$, hence
there exists a topological vector bundle $F$ over $X$ such that $E\oplus F$ is a topologically trivial vector bundle over $X$.
By the Oka-Grauert principle, there is a unique holomorphic structure on $F$ and $E\oplus F$ is trivial as a holomorphic vector bundle over $X$.
\end{proof}


\begin{lem}\label{metric}
For any vector bundle $E$ on a Stein manifold $X$,
there is a smooth hermitian metric $h$ on $E$
such that $i\Theta_h$ is both Nakano positive and dual Nakano positive.
Furthermore, there is a smooth hermitian metric $h'$ on $E$
such that $i\Theta_{h'}$ is  Nakano negative.
\end{lem}

\begin{proof}
We only need to prove the first statement since $E$ is isomorphic to its dual bundle $E^*$ as holomorphic vector bundles, by the Oka-Grauert principle.
We fix a smooth hermitian metric $h_0$ on $E$ at first,
and a exhaustive smooth strictly  plurisubharmonic function $\phi$ on $X$.
Then we can choose a smooth increasing convex function $u$, such that
the curvature of $h=e^{-u(\phi)}h_0$ is both Nakano positive and dual Nakano positive.
\end{proof}

We need the following proposition.
\begin{prop}\label{dec}[see Proposition 6.10 of \cite{D b}, P340]
Let $0\rightarrow S\rightarrow E\rightarrow Q\rightarrow 0$ be an exact sequence of hermitian
vector bundles, and $h$ is a smooth hermitian metric on $E$.  Give $S$ and $Q$ the induced metric by $E$.
Then

(a) If  $i\Theta_E$ is Griffith (semi-)positive, so is $i\Theta_Q$.

(b) If $i\Theta_E$ is Griffith (semi-)negative, so is $i\Theta_S$.

(c) If $i\Theta_E$ is Nakano (semi-)negative, so is $i\Theta_S$.

\end{prop}

B. Berndtsson and M. P\v{a}un \cite{BP} proved the following theorem in the case that
$h$ is negatively curved in the sense of Griffiths by convolution in an approximate way.
H. Raufi \cite{R} observed that the same technique yields
a similar result when $h$ is negatively curved in the sense of Nakano.

\begin{thm}[see \cite{BP} and \cite{R}]\label{BPR}
Let $h$ be a singular hermitian metric on a trivial vector bundle $E$ over a polydisc,
and assume that $h$ is negatively curved in the sense of Griffiths (resp. Nakano).
There exists a sequence of smooth hermitian metrics $\{h_\nu\}_{\nu=1}^\infty$
with negative curvature in the sense of Griffiths (resp. Nakano), decreasing to $h$ pointwise on any smaller polydisc.
\end{thm}

%
%
%

We now give the proof of Theorem \ref{thm c g}.
\begin{thm}[=Theorem \ref{thm c g}]
Let $(E,h)$ be a singular hermitian holomorphic vector bundle over a Stein manifold $X$.
Assume $h$ is negatively curved in the sense of Griffiths (resp. Nakano).
Then there exists a sequence of smooth hermitian metrics $\{h_k\}_{k=1}^\infty$ with negative curvature in the sense of Griffiths (resp. Nakano),
such that for any compact subset $K\subset X$, there is $k(K)$, for $k\geq k(K)$, $h_k$ decrease to $h$
pointwise on $K$.
\end{thm}

\begin{proof}
We give the proof for the case of Griffiths negativity.
The proof for the case of Nakano negativity is similar.

From Lemma \ref{trivial}, there is a holomorphic vector bundle $F$ on $X$,
such that $E\oplus F$ is a trivial holomorphic vector bundle.
As the restriction of trivial hermitian metric of trivial vector bundle on $F$
 which denoted by $g$ is Griffiths negative,
so $(h,g)$ is negatively curved in the sense of Griffiths.
If we get an approximation of $(h,g)$ by hermitian metrics with Griffiths negative curvatures,
by Proposition \ref{dec}, the restriction of the hermitian metrics to $E$ which are Griffiths negative
is an approximation of $h$.
Without of loss generality, we may assume that $E$ is the trivial vector bundle $X\times \mathbb{C}^r$,
and $h$ is a hermitian matrix function on $X$.

By the imbedding theorem, there is a proper holomorphic imbedding $\iota:X\rightarrow\mathbb{C}^N$.
So we may view $X$  as a closed complex submanifold of $\mathbb{C}^N$.
There exist finitely many holomorphic functions $f_1,\cdots,f_m$ on $\mathbb{C}^N$,
such that $\{z\in\mathbb{C}^N:f_1(z)=\cdots=f_m(z)=0\}=X$. Set $|f|^2=\sum_{j=1}^{m}|f_j|^2$.
By Cor. 1 of \cite{S}, there is a Stein open neighborhood $U\subset\mathbb{C}^N$ of $X$,
and a holomorphic retraction $\pi:U\rightarrow X$.
Hence, $\pi^*(h)=h\circ\pi$ is a singular hermitian metric on $U\times \mathbb{C}^r$,
 which is an extension of $h$ and negatively curved in the sense of Griffiths.

We will construct $h_k$ by induction.
For $z\in\mathbb{C}^N$, set $|z|=\sqrt{\sum_{j=1}^{N}|z_j|^2}$, and
$X_k=\{z\in X:|z|\leq k\}$. Denote $d_k=dist(X_k,\partial U)$.
Let $\rho$ be a smooth radical function on $\mathbb{C}^N$, $\rho\geq 0$,
 $\text{supp}\ \rho\subset B_1$ and $\int_{\mathbb{C}^n}\rho dV=1$.
Set $\rho_\epsilon(z):=\frac{\rho(z/\epsilon)}{\epsilon^{2N}}$ for $\epsilon>0$.
Hence by Theorem \ref{BPR}, $(h\circ\pi)\ast\rho_{\epsilon_k}$ is Griffiths negative on
$$V_k:=\{z\in\mathbb{C}^N:dist(z,X_k)<d_k-\epsilon_k\}$$ for $\epsilon_k<d_k$.
We may choose $\{\epsilon_k\}$ inductively, such that $\epsilon_k$ is decreasing to 0,
as $k\rightarrow\infty$.
Let $\chi_k$ be a smooth function, such that $\chi_k\equiv 1$ in a neighborhood $U_k$ of $X_k$,
and $\text{supp}\chi_k\subset V_k$. Let $\delta_k=\frac{1}{4}dist(X_k,\partial U_k)>0$,
and $$u_k(z)=max_{\delta_k}\{|z|^2-k^2-\delta_k,0\}+|f(z)|^2,$$
where $\max_{\delta}\{x,y\}$ means the convolution of the maximum function (see \cite{D b}).
Notice that
\begin{equation}\label{in1}
 u_k|_{X_k}=0\; \text{and}\;u_k(z)=|z|^2-k^2-\delta_k+|g(z)|^2\geq c_k,\; \text{for}\; z\in U_k^c
\end{equation} for some constant $c_k>0$.

\textbf{Claim:} There is a sequence of smooth increasing convex functions $\{v_k\}$, with $v_k(0)=0$,
such that $$\tilde{h}_k:=e^{v_k(u_k)}(\chi_k(h\circ\pi)\ast\rho_{\epsilon_k}+\epsilon_k I_r)$$
is a Griffiths negative hermitian metric for the trivial vector bundle $\mathbb{C}^N\times\mathbb{C}^r$.

If the claim is proved, let $h_k=\tilde{h}_k|_X$, $h_k$ is Griffiths negative,
and $$h_k|_{X_k}=(h\circ\pi)\ast\rho_{\epsilon_k}|_{X_k}+\epsilon_k I_r\; \text{for}\; k\geq1.$$
As $\epsilon_k$ decreases to 0, then $\{h_k\}$ satisfies all the conditions.

We will prove the claim.

Note that,
$$i\Theta_{h_k}=i\Theta_{\chi_k(h\circ\pi)\ast\rho_{\epsilon_k}|_{X_k}+\epsilon_k I_r}-i\partial\bar{\partial}v_k(u_k)\otimes I_r.$$
From Proposition \ref{add},
$\chi_k(h\circ\pi)\ast\rho_{\epsilon_k}|_{X_k}+\epsilon_k I_r$ is Griffiths negative on $U_k$.
As $\chi_k=0$ on $V_k^c$, we only need to consider $V_k\backslash U_k$.
Since $\chi_k(h\circ\pi)\ast\rho_{\epsilon_k}|_{X_k}+\epsilon_k I_r$ is a smooth hermitian metric,
$\overline{V_k}\backslash U_k$ is compact,
and $u_k$ is strictly plurisubharmonic on $\overline{V_k}\backslash U_k$,
we can find a smooth increasing convex function $v_k$, such that $v_k(0)=0$
and $\tilde{h}_k$ is Griffiths negative.

\end{proof}

\begin{rem}
  By the dual proposition, we can also get similar result for the approximation of singular metric
 which is
  positively curved in the sense of Griffiths or dual Nakano positive on Stein manifold
   with increasing sequence of metrics with Griffiths positive or dual Nakano positive curvatures.
\end{rem}




\section{extensions of hermitian metrics}

In this section we give the proof of Theorem \ref{thm ext G}.

\begin{thm}[=Theorem \ref{thm ext G}]\label{thm ext G 1}
Let $E$ be a  holomorphic vector bundle over a Stein manifold $X$, and $Y$ be a closed submanifold of $X$.
Assume that $h$ is a smooth hermitian metric on $E|_Y$ with negative curvature in the sense of Griffiths (resp. Nakano),
then $h$ can be extended to a smooth hermitian metric $\tilde{h}$ on $E$  with negative curvature in the sense of Griffiths (resp. Nakano).
\end{thm}

\begin{proof}
We give the proof for the case of Griffiths negativity.
The proof for the case of Nakano negativity is similar.

From Lemma \ref{trivial}, there is a holomorphic vector bundle $F$ on $X$,
such that $E\oplus F$ is the trivial holomorphic vector bundle $X\times\mathbb{C}^r$.
The restriction of trivial hermitian metric of trivial vector bundle on $F$
 which denoted by $g$ is Griffiths negative.

As $Y$ is a closed submanifold of the Stein manifold $X$,
there exists an open neighborhood $U$ of $Y$ in $X$,
with a holomorphic retraction $\pi:U\rightarrow Y$.
So we get a smooth hermitian metric $(h\circ\pi,g)$ on $(E\oplus F)|_U$ with negative Griffiths curvature.

Let $\chi$ be a smooth function satisfying $\chi\equiv1$ in an open neighborhood $V$ of $Y$
and $\text{supp}\ \chi\subset U$. Set $$h'=\chi(h\circ\pi,g)+(1-\chi)I_r,$$
then $h'$ is a smooth Hermitian metric on $E\oplus F$.
As $h'|_V=(h\circ\pi,g)$, $h'$ is Griffiths negative on $V$.
So we will focus on $V^c:=X\setminus V$.

There exist finitely many holomorphic functions $f_1,\cdots,f_m$ on $X$,
such that $$Y=\{z\in X:f_1(z)=\cdots=f_m(z)=0\}.$$
Let $|f|^2=\sum_{j=1}^{m}|f_j|^2$, and
$\phi\geq0$ be a smooth strictly plurisubharmonic exhaustion function on $X$.
There is a continuous function $v$ on $[0,\infty)$, such that
\begin{equation}\label{ex in1}
 i\Theta_{h'}\leq v(\phi)i\partial\bar{\partial}\phi I_r
\end{equation}
and
\begin{equation}\label{ex in2}
\frac{1}{|f|^2}\leq v(\phi) \;\;\text{on}\;\; V^c.
\end{equation}

Let $u$ be a smooth increasing convex function on $(-1,\infty)$,
such that
\begin{equation}\label{ex in3}
  u(t)\geq0 \;\;\text{and}\;\;u'(t)\geq v^2(t)\;\;\forall t\in[0,\infty).
\end{equation}

Let $$\tilde{h}=e^{|f|^2e^{u(\phi)}}h',\;\;
\text{and}\;\; \psi=\log|f|^2+u(\phi).$$
Then $\tilde{h}$ is a smooth hermitian metric on $E\oplus F$, and
$$\tilde{h}|_Y=h'|_Y=(h,g|_Y).$$
Since $\psi$ is plurisubharmonic,
we have that $e^{\psi}=|f|^2e^{u(\phi)}$ is also plurisubharmonic.
As
\begin{equation}\label{ex in4}
   i\Theta_{\tilde{h}}
   =i\Theta_{h'}-i\partial\bar{\partial}(|f|^2e^{u(\phi)})I_r,
\end{equation}
and $h'$ is Griffiths negative on $V$,
Therefore, $\tilde{h}$ is also Griffiths negative on $V$.
On $V^c$, notice
\begin{equation}\label{ex in5}
\begin{split}
  i\partial\bar{\partial}(|f|^2e^{u(\phi)})
  =&i\partial\bar{\partial}e^{\psi}\\
  =&e^{\psi}(i\partial\bar{\partial}\psi+i\partial\psi\wedge\bar{\partial}\psi)\\
  \geq&|f|^2e^{u(\phi)}i\partial\bar{\partial}(\log|f|^2+u(\phi))\\
  \geq&|f|^2e^{u(\phi)}u'(\phi)i\partial\bar{\partial}\phi\\
  \geq&v(\phi)i\partial\bar{\partial}\phi,
  \end{split}
\end{equation}
the last inequality holds form inequality \ref{ex in2} and \ref{ex in3}.

From inequality \ref{ex in1}, \ref{ex in4} and \ref{ex in5},
we have that $\tilde{h}$ is Griffiths semi-negative on $V^c$.
Therefore, $\tilde{h}$ is is Griffiths semi-negative on $X$.
Let $$\iota:E\rightarrow E\oplus F,\;\;\iota(\alpha)=(\alpha,0).$$
$\iota^*(\tilde{h})$ is a smooth Griffiths semi-negative hermitian metric
 on $E$ by Proposition \ref{dec}, with $\iota^*(\tilde{h})|_Y=h$.

\end{proof}

%

\begin{rem}
For the extension theorems, if the metric being extended is Griffiths nagative or Nakano negative,
we can get a Griffiths negative or Nakano negative extension metric. The proof is almost same as above,
if one notice that there is a sequence of holomorphic functions $\{f_j\}$ on $X$,  such that
$|f|^2:=\sum_{j=1}^{\infty}|f_j|^2$ converges uniformly on compact subsets of $X$, $Y=\{|f|^2=0\}$,
and $i\partial\bar{\partial}|f|^2+\omega_Y>0$ on every point $p\in Y\subset X$ for any given K\"{a}hler form $\omega_Y$ on $Y$.
\end{rem}

\begin{thm}\label{thm ext G 2}
Let $E$ be a  holomorphic vector bundle over a Stein manifold $X$, and $Y$ be a closed submanifold of $X$.
Assume $h$ is a smooth hermitian metric on $E|_Y$ which is Nakano (semi-)positive.
Then $h$ can be extended to a smooth hermitian metric $\tilde{h}$ on $E$ which is Nakano (semi-)positive.
\end{thm}

\begin{proof}
As $Y$ be a closed submanifold of the Stein manifold $X$, by \cite{S}, there is a Stein 
neighborhood $U\subset X$ of $Y$, and a holomorphic retraction $\rho:U\rightarrow Y$.
By the Oka-Grauert principle, $\rho^*(E|_Y)$ is isomorphic to $E|_U$, and $\rho^*(E|_Y)|_Y=E|_Y$.
Therefore, we may view $\rho^*(h)$ as a smooth hermitian metric on $E|_U$, and $\rho^*(h)|_Y=h$. 
Choose a smooth hermitian metric $h_1$ on $E$, and a smooth function $\chi$ on $X$
such that $\chi=1$ in an open neighborhood of $Y$, and $\text{supp}\chi\subset U$.
Then we get a smooth hermitian metric $$h'=\chi\rho^*(h)+(1-\chi)h_1$$ on $E$.

There is a sequence of holomorphic functions $\{f_j\}$ on $X$,  such that
$|f|^2:=\sum_{j=1}^{\infty}|f_j|^2$ converges uniformly on compact subsets of $X$, $Y=\{|f|^2=0\}$,
and $i\partial\bar{\partial}|f|^2+\omega_Y>0$ on every point $p\in Y\subset X$ for any K\"{a}hler form $\omega_Y$ on $Y$.

By a similar procedure of the proof of theorem \ref{thm ext G}, we can find a smooth
plurisubharmonic function $\varphi$ on $X$, such that $e^{-|f|^2e^\varphi}h'$ is Nakano (semi)-positive
on $X$. It is obvious $e^{-|f|^2e^\varphi}h'$ is an extension of $h$.

\end{proof}

\bibliographystyle{amsplain}

	\end{document}